\documentclass{article}
\usepackage{amssymb}
\usepackage{svg}
\usepackage{amsthm}
\usepackage{xcolor}
\usepackage{comment}
\usepackage{graphicx}
\usepackage{float}
\usepackage{setspace}
\usepackage{verbatim}
\usepackage[utf8]{inputenc}
\usepackage[russian]{babel}
\usepackage{fancyhdr}
\usepackage{mathtools}
\newtheorem{lemma}{Лемма}
\newtheorem{teor}{Теорема}

\newtheorem{Prop}{Предложение}
\newtheorem{df}{Определение}

\def\Z{{\mathbb Z}}

\title{Изопериодическое слоение на пространствах модулей вещественно-нормированных мероморфных дифференциалов с одним полюсом}
\author{М.~Ненашева\thanks{Сколковский Институт Науки и Технологий, Национальный Исследовательский Институт Высшая Школа Экономики Работа поддержана грантом РНФ 23-11-00150 ``Математические проблемы современной математической физики''}}
\date{\today}
\begin{document}
\maketitle

\section{Введение}

\subsection{Пространства модулей мероморфных дифференциалов}
Всякий ненулевой голоморфный абелев дифференциал на римановой поверхности определяет на ней
плоскую метрику всюду кроме конечного числа точек. 
Исключительные (проколотые) точки это нули абелева дифференциала, в них метрика имеет конические особенности.
Пусть $(X,\psi)$ и $(X',\psi')$ --- две римановы поверхности с заданными на них голоморфными дифференциалами.  Если биголоморфизм $f:X\rightarrow X'$ такой, что $f^*\psi'=\psi$, то $f$ является изометрией для метрик, определенных $1$-формами  $\psi,\psi'$. В локальных координатах, заданных
$1$-формами $\psi,\psi'$, отображение $f$ является переносом.

Рассмотрим пространство, точками которого являются классы эквивалентности пар (риманова поверхность фиксированного рода $g$, голоморфный дифференциал на ней), $(X,\psi)\sim (X', \psi')$, если существует биголоморфизм $f:X\rightarrow X'$, такой, что $f^*\psi'=\psi$.  Такое пространство называется пространством модулей голоморфных дифференциалов рода $g$ и обозначается $\mathcal{H}_g$. Оно наделено естественной структурой
комплексного орбифолда, см., например,~\cite{qd}.

Задачи, относящиеся к геометрии пространств $\mathcal{H}_g$, возникают в тейхмюллеровой геометрии, геометрии бильярдов на многоугольниках,
 при изучении перекладываний отрезков и многих других областях.

Голоморфный дифференциал на компактной римановой поверхности рода~$g$ имеет $2g-2$ нулей с учетом кратностей. Разбиения числа $2g-2$ определяют стратификацию пространства $\mathcal{H}_g$: страт $\mathcal{H}(\kappa)\subset \mathcal{H}_g$, заданный разбиением $\kappa\vdash 2g-2$, $\kappa=
(k_1, \dots, k_n)$, $k_i \in \mathbb{N}$, $k_1+\dots+k_n=2g-2$, состоит из точек, для которых соответствующие голоморфные дифференциалы имеют нули кратностей $k_1, \dots, k_n$. 
Известно, что замыкание всякого страта $\mathcal{H}(\kappa)$ в $H_g$ является алгебраическим многообразием и комплексным орбифолдом  размерности $2g + n-1$, см., например,~\cite{qd}.

Страт $\mathcal{H}(\kappa)$ может оказаться несвязным, что было впервые отмечено У. Вичем~\cite{qd}.  Полное описание компонент связности произвольных стратов $\mathcal{H}(\kappa)$ дано в работе Концевича и Зорича \cite{KZ}. В частности, известно, что число компонент связности не превышает трех. Страты
$\mathcal{H}(\kappa)$ не компактны, поскольку не компактно пространство модулей гладких кривых данного рода.

Пара (риманова поверхность, мероморфный дифференциал на ней) называется \emph{плоской поверхностью с полюсами}.  Аналогично ситуации с голоморфными дифференциалами, пространства модулей мероморфных дифференциалов с полюсами фиксированных порядков допускают стратификацию по наборам кратностей нулей. Мы будем обозначать подмногообразие  мероморфных дифференциалов с $r$ нулями кратностей $n_1,\dots,n_r$ в пространстве дифференциалов с $k$ полюсами данных фиксированных порядков $m_1,\dots, m_k$ через $\mathcal{H}(n_1,\dots,n_r;-m_1,\dots -m_k)$.
Как следует из теоремы 1.1 в \cite{Bo}, в роде один существуют пространства мероморфных дифференциалов с произвольным количеством компонент связности, в то время как для рода два и выше у такого пространства не может быть более трех компонент связности (теорема 1.2 в \cite {Bo}). В настоящей работе мы рассматриваем пространства вещественно-нормированных мероморфных дифференциалов с единственным полюсом порядка~$2$, которые мы обозначаем $\mathcal{R}_g$; они являются подпространствами в пространстве $\mathcal {H}(1^{2g},-2)$.

 \subsection{Вещественно-нормированные дифференциалы}
Интеграл мероморфного дифференциала по замкнутой кривой на римановой поверхности называется \emph{периодом} этого дифференциала.

\begin{df}
Мероморфный дифференциал на римановой поверхности называется \emph{вещественно-нормированным}, если все его периоды вещественны.
\end{df}

В частности, вычеты вещественно-нормированного дифференциала в каждом из его полюсов являются чисто мнимыми.

\begin{df}
\emph{Главная часть} мероморфного дифференциала в точке  $P$ на римановой поверхности $X$ это класс эквивалентности мероморфных дифференциалов $\psi$, определенных в окрестности точки $P$, относительно
следующего отношения эквивалентности: $\psi\sim \psi'$ тогда и только тогда, когда $\psi'-\psi$ голоморфен в точке $P$.
\end{df}

\begin{Prop}{\rm\cite{GK}}
Пусть $(X;x_1,\dots,x_n;p_1,\dots,p_n)$
--- риманова поверхность с~$n>0$ попарно различными отмеченными точками $x_1,\dots,x_n$, в каждой из которых задана
главная часть $p_1,\dots,p_n$ мероморфного дифференциала, причем вычеты всех этих главных частей чисто мнимы, а их сумма равна~$0$.
Тогда на~$X$ существует единственный вещественно-нор\-ми\-ро\-ван\-ный мероморфный дифференциал, имеющий главные части~$p_i$ в своих полюсах~$x_i$.
\end{Prop}

В частности, на всякой римановой поверхности~$X$ существует единственный вещественно-нор\-ми\-ро\-ван\-ный дифференциал с полюсом $2$-го порядка  в произвольной заданной точке $x_1\in X$ с заданной главной частью~$p_1$, имеющей нулевой вычет.

\subsection{Изопериодическое слоение  на пространстве вещественно-нормированных дифференциалов $\mathcal{R}_g$}

Множество периодов данного мероморфного дифференциала на римановой поверхности образует группу по сложению. Эта группа является подгруппой в группе комплексных чисел по сложению. В свою очередь, группа периодов вещественно-нормированного мероморфного дифференциала является подгруппой в группе вещественных чисел по сложению.

Голоморфный дифференциал на римановой поверхности называется \emph{максимально несоизмеримым}, если из того, что его интеграл по некоторому циклу равен~$0$, вытекает, что этот цикл гомологичен нулю на поверхности. Группа периодов такого голоморфного дифференциала естественно отождествляется с группой целочисленных гомологий $H_1(X;\Z)$. Форма пересечений на гомологиях задает невырожденную симплектическую форму на группе периодов дифференциала.

Подгруппа  в $\mathbb{C}$, изоморфная целочисленной решетке, с заданной на ней невырожденной симплектической формой,
называется \emph{поляризованным модулем}.

Таким же образом мы можем ввести структуру поляризованного модуля, обозначим его $Per(X,\psi)\subset\mathbb{R}$, на группе периодов макси\-мально-несоизмеримого мероморфного вещественно-нормированного дифференциала с единственным полюсом второго порядка.
Будем говорить, что точка рассматриваемого пространства модулей дифференциалов, имеющая представитель $(X,\psi)$, также имеет данный модуль периодов $Per(X,\psi)$.

Пусть $L$ --- поляризованный модуль ранга~$2g$; через $\mathcal{A}_L$ будем обозначать
подпространство в пространстве модулей вещественно-нормированных дифференциалов с единственным полюсом второго порядка, состоящее из точек с модулем периодов $L$: $$\mathcal{A}_L=\{(X,\psi): Per(X,\psi)\cong L \mbox{ как поляризованный модуль}\}.$$

В недавней работе~\cite{iso3} Кальсамильи и Деруана такие подпространства рассматриваются в  пространстве модулей $\mathcal{H}{(1,1;-2)}$ мероморфных дифференциалов с одним полюсом порядка два и двумя простыми нулями.

Листами \emph{изопериодичного слоения} являются компоненты связности пространств $\mathcal{A}_L$.
Стратификация кратностями нулей на пространстве мероморфных дифференциалов индуцирует стратификацию на $\mathcal{A}_L$.  Страт  $\mathcal{A}_L\cap \mathcal{H}(n_1,\dots,n_r)$ будем обозначать через $\mathcal{A}_L(n_1,\dots,n_r)$.

\subsection{Основная теорема}

В этой работе мы рассматриваем подпространства $\mathcal{A}_L$ в пространствах $\mathcal{R}_g$ ве\-щест\-вен\-но-нормированных мероморфных дифференциалов с единственным полюсом порядка два  на кривых рода~$g$
в случаях, когда все периоды несоизмеримы,
то есть когда поляризованный
модуль периодов  вещественно-нормированного дифференциала $L$ представляет собой решетку ранга $2g$, и ее гомоморфизм в $\mathbb{R}$ является вложением. 

Слои изопериодческого слоения  в стратах старшей размерности в пространстве вещественно-нор\-ми\-ро\-ванных дифференциалов с единственным полюсом второго порядка
  изучались в~\cite{LK}. В частности, там доказано, что любые две точки страта коразмерности~$0$, имеющие одну и ту же группу
  периодов ранга $2g$, но такие, что  эти  группы по-разному поляризованы, нельзя соединить непрерывным путем вдоль которого группа периодов не меняется. Мы доказываем,
  что, точки пространства модулей с одинаково поляризованными группами периодов лежат в одном слое,
   давая, тем самым, полное описание компонент линейной связности изопериодических слоев в стратах старшей размерности.

\begin{teor}
Если модуль периодов~$L$ максимально несоизмерим, то страт $\mathcal{A}_L(1^{2g})$ в пространстве $A_L$ вещественно-нормированных мероморфных дифференциалов с единственным полюсом второго порядка
и данным поляризованным модулем периодов~$L$ на кривых рода~$g$
имеет ровно одну компоненту линейной связности.

\label{teorema}
\end{teor}

\subsection{План статьи}
Работа имеет следующую структуру:

В разделе~\ref{s2} мы кратко описываем комбинаторную модель пространства модулей вещественно-нормированных дифференциалов, предложенную в работе И.~Кри\-чевера,  C.~Ландо, А.~Скрипченко \cite{LK}.
Эта модель служит в дальнейшем нашим основным инструментом.
Раздел~\ref{s3} посвящен доказательству  основной теоремы.

\section{Комбинаторная модель пространства модулей}\label{s2}

\subsection{Стратификация пространства модулей вещественно-нормированных дифференциалов}
Пусть на комплексной кривой $X$ рода $g$ задан вещественно-нормированный дифференциал $\psi$ c единственным полюсом порядка два.  Каждой точке~$q\in X$, не являющейся полюсом или нулем дифференциала $\psi$, сопоставим вещественную прямую в касательной плоскости $T_q X$ к кривой $X$ в этой точке --- росток касательной к кривой, заданной уравнением $Re(\int_q^z\psi)=0$. Ориентируем прямую в направлении, в котором мнимая часть интеграла $\int_q^z\psi$ увеличивается. Заданные таким образом прямые  образуют поле направлений на $X$, определенное везде, кроме конечного числа точек --- нулей и полюса $1$-формы $\psi$.
Обозначим это поле направлений через $V_\psi$. Его можно считать полем направлений
с особенностями, заданным на всей кривой~$X$.

Пусть точка $\mathcal{O}$ --- нуль дифференциала $\psi$ порядка $k$. Тогда в малой окрестности этой точки $U(\mathcal{O})$ существует локальная координата $z$, такая, что дифференциал, ограниченный на эту окрестность, принимает вид $\psi(z)=z^kdz$. Это означает, что в точке $\mathcal{O}$ имеется $2k+2$ направлений, таких, что $Re(\int_{\mathcal{O}}^z z^kdz)=0$; они соответствуют корням из $i,-i$ степени $k+1$. Интегральные кривые, отвечающие направлениям, соответствующим корням из $-i$, являются входящими в~$\mathcal{O}$, а соответствующие корням из $i$
--- исходящими из~$\mathcal{O}$.

Интегральные кривые поля направлений $V_\psi$,
входящие в нули дифференциала $\psi$  (или выходящие из них), называются его \emph{сепаратрисами}.   В этом разделе мы предполагаем, что все нули дифференциала $\psi$ простые. Предположим также, что
никакие два нуля не соединены сепаратрисой. Пример возможного расположения сепаратрис на кривой рода $1$ показан на рисунке \ref{ex1}.\\
\begin{figure}[H]
    \centering
    \includegraphics[scale=0.6]{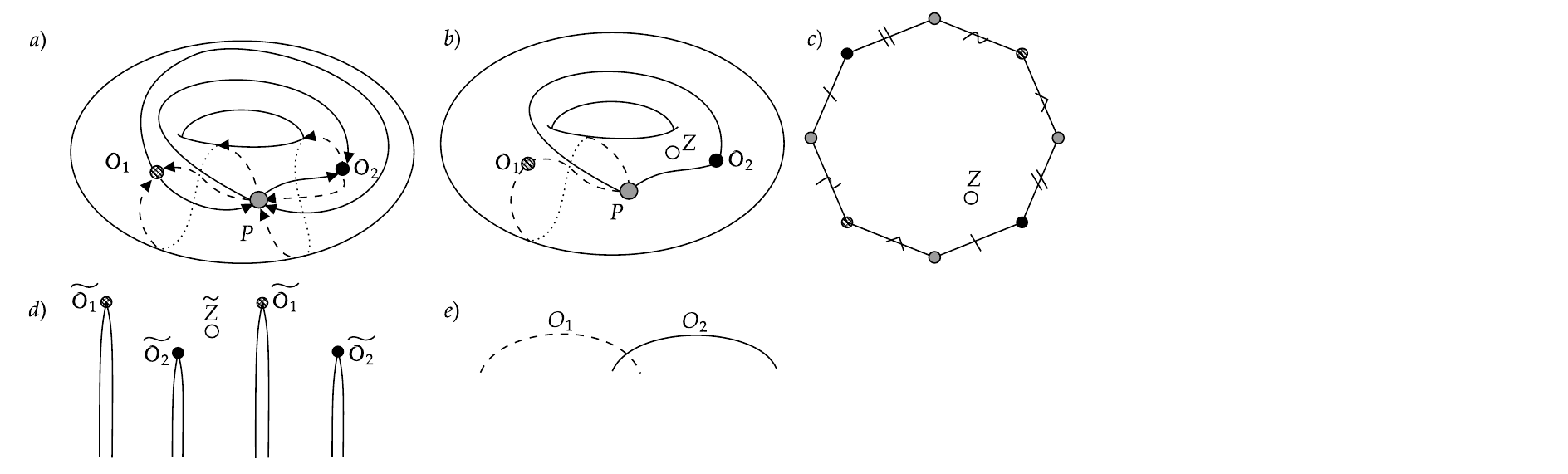}
    \caption{Пример: a) сепаратрисы  на поверхности рода~$1$; b) граф сепаратрис; c) $8$-угольник, являющийся результатом разрезания поверхности по графу сепаратрис;
    d) система разрезов комплексной прямой;
    e) дуговая диаграмма, отвечающая этой системе разрезов}
    \label{ex1}
\end{figure}
\subsubsection{Комбинаторное описание}
Сопоставим вещественно-нормированному дифференциалу $\psi$ на кривой $X$ рода $g$ граф $G_{\psi}$ на $X$. Вершинами этого графа являются $2g$ простых нулей и полюс дифференциала $\psi$, а ребрами --- все входящие в нули сепаратрисы. Поскольку все нули простые, в каждый из них входят две сепаратрисы.

 Никакая из сепаратрис не может иметь началом и концом один и тот же нуль.  По нашему предположению также никакие два нуля
 не соединены сепаратрисой. Кроме того, предельной точкой сепаратрисы при движении по ней в направлении, противоположном
 направлению поля, может быть только полюс $1$-формы~$\psi$.
 Действительно, для любой точки $A\in X$
 функция $I_{\psi,A}(q)={\rm Im}\int_A^q\psi$ 
 является однозначно 
 определенной гармонической функцией на поверхности~$X$.
 Поэтому у поля направлений $V_\psi$ нет предельных циклов:
 при движении вдоль предельного цикла функция $I_{\psi,A}$ 
 должна была бы монотонно возрастать.
 Таким образом, предельными точками траекторий поля $V_\psi$
 могут быть только особые точки этого поля, а значит, предельной
 точкой сепаратрисы нуля при движении по ней в направлении, противоположном
 направлению поля, является полюс $1$-формы~$\psi$. 
 Значит все $4g$ ребер соединяют нули с полюсом.
 Каждая пара сепаратрис, входящих в один нуль~$Q$, образует нестягиваемую петлю
 с началом и концом в точке~$Q$, представляющую нетривиальный гомотопический класс.

 Поскольку вложенный в~$X$ граф  $G_{\psi}$ имеет $4g$ ребер и $2g+1$ вершин, по формуле Эйлера он не разбивает поверхность. Результатом
 разрезания поверхности~$X$ по графу является $8g$-угольник,  каждая вторая вершина которого соответствует полюсу дифференциала~$\psi$.

Определим отображение $F_\psi$ поверхности~$X$ с выкинутыми из нее ребрами графа~$G_\psi$ в комплексную прямую с  $4g$ разрезами следующим образом: отметим точку~$Z$ на поверхности $X$, отличную от вершин вложенного графа $G_{\psi}$ и не лежащую ни на одном из его ребер.
Положим $F_\psi:q\mapsto\int_Z^q\psi$,
где интегрирование идет по произвольному пути, не пересекающему
ребер графа~$G_\psi$.

Отображение~$F_\psi$ имеет два различных предела при подходе к точкам ребер графа $G_\psi$, в зависимости от того, с какой стороны мы подходим к сепаратрисе. Совокупность этих пределов образует $4g$ разрезов комплексной прямой вдоль вертикальных лучей, направленных вниз. Обозначим полученную систему разрезов $S_{\psi}$.

Эти разрезы разбиты на пары, образованные сепаратрисами, входящими в один и тот же нуль дифференциала.

Заметим также, что разность координат двух образов нуля дифференциала~$\psi$, --- вершин парных разрезов --- равна значению интеграла от $\psi$ по некоторому циклу с началом и концом в этом нуле. Поскольку дифференциал $\psi$  вещественно-нормированный, образы одного нуля оказываются на одном горизонтальном уровне.

На рис.~\ref{ex1} изображен многоугольник, результатом склейки соответственных сторон
которого является поверхность рода один, и соответствующий набор разрезов комплексной прямой. Образ точки $Z$  на комплексной прямой  обозначен $\tilde{Z}$; вершины, соответствующие образу одного и того же нуля $\mathcal{O}_i$, обозначены одинаковыми символами $\widetilde{\mathcal{O}_i}$.

Заметим, что выбрав в качестве начальной любую другую точку $Z'$ поверхности~$X\setminus G_\psi$ вместо $Z$, мы получим такой же набор разрезов комплексной прямой, но сдвинутый на комплексный вектор, равный значению $\int_Z^{Z'}\psi$, где интегрирование идет по произвольному пути, не пересекающему ребер графа~$G_\psi$. Поэтому система вертикальных разрезов на~$\mathbb{C}$, отвечающих дифференциалу~$\psi$, определяется
однозначно с точностью до сдвига на комплексное число.

\emph{Дуговой диаграммой} называется горизонтальная прямая на плоскости вместе с четным числом точек на ней, разбитых на пары;
на рисунках точки одной пары мы соединяем полуокружностями в верхней полуплоскости с концами в этих точках.

Набору $S_{\psi}$ разрезов комплексной прямой, полученному при разрезании поверхности $X$,  можно сопоставить следующую дуговую диаграмму $\mathcal{D}_{\psi}$. Проведем в~${\mathbb C}$ горизонтальную прямую, лежащую ниже начальных точек всех разрезов. Каждому разрезу сопоставим его точку пересечения с прямой, получим  $4g$ точек. Разбиение разрезов на пары определяет
разбиение на пары этих точек. Результатом является $2g$ дуг, образующих \emph{дуговую диаграмму} $\mathcal{D}_{\psi}$ вещественно-нормированного дифференциала~$\psi$.

Наоборот, набор $S$ не перекрывающихся разрезов комплексной прямой вдоль вертикальных лучей, направленных вниз, разбитых на пары, начинающиеся на одной высоте, определяет риманову поверхность с мероморфным вещественно-нормированным дифференциалом на ней, имеющим простые нули: поверхность является  результатом склейки  сторон каждого разреза с противоположными сторонами парного к нему разреза, а дифференциал имеет вид $dz$, где $z$ --- координата на комплексной прямой. Два таких набора разрезов определяют одну и ту же точку (класс эквивалентности римановых поверхностей с вещественно-нормированным дифференциалом) пространства модулей в том и только в том случае, если один из них получается из другого сдвигом на комплексное число. Если мы хотим, чтобы  это был вещественно-нормированный мероморфный дифференциал на кривой рода~$g$ с единственным полюсом порядка~$2$, то количество разрезов в $S$ должно равняться~$4g$ и должно выполняться описанное ниже простое комбинаторное условие, гарантирующее наличие единственного полюса.

Проще всего сформулировать это условие в терминах дуговой диаграммы $\mathcal{D}$, построенной по $S$ так, как это было описано выше.
Пусть дуги диаграммы занумерованы в порядке возрастания вещественной части их левых концов. Определим матрицу пересечений $I(\mathcal{D})$ дуговой диаграммы $\mathcal{D}$ как квадратную матрицу размером $2g\times 2g$,
элементами которой являются числа $-1,0,1$, причем на пересечении $i$-й строки и $j$-го столбца стоит~$0$, если дуги с номерами $i$ и $j$ не пересекаются (или если $i=j$), $-1$, если дуги пересекаются и $i>j$, и $1$, если дуги пересекаются и $i<j$. Построенная матрица пересечений кососимметрична. Как хорошо известно (см., например, \cite{sob}), поверхность, полученная склейкой плоскости по  системе разрезов $S$, будет иметь одну точку на бесконечности (один полюс мероморфного дифференциала) в том и только в том случае, когда определитель матрицы пересечения нечетен.  Заметим, что четность определителя не меняется при замене нумерации дуг.

Дуговые диаграммы, удовлетворяющие этому свойству, называются \emph{допустимыми}.

\emph{Сдвигом} называется преобразование дуговой диаграммы, при котором одна из дуг с концами в точках $x_1,x_2$ заменяется на новую, c концами в точках $x_1+\varepsilon,x_2+\varepsilon$, где $\varepsilon$ достаточно маленькое комплексное число. ``Достаточно маленькое'' означает
здесь, что порядок концов всех дуг на полученной диаграмме сохраняется. Сдвиг оставляет диаграмму допустимой.

Путь, соединяющий две точки пространства модулей в одном листе
изопериодического слоения, может проходить через точки,
отвечающие дифференциалам, поле направлений которых имеет одну сепаратрису, соединяющую
два различных нуля. При прохождении через такую точку
комбинаторика дуговой диаграммы дифференциала меняется.
Это преобразование известно из теории инвариантов узлов
конечного порядка, в которой оно называется \emph{второе движение Васильева}.  Оно состоит в замене одной из дуг в паре дуг, имеющих  соседние концы, обозначим их $O_i,O_j$, результатом протягивания ее конца вдоль второй дуги, как это показано на  рис. \ref{vm}: сплошная дуга сохраняется, тогда как дуга, 
изображенная пунктиром, заменяется. При этом движении допустимая
дуговая диаграмма переходит в допустимую.
Поскольку в дальнейшем мы будем пользоваться только вторым движением Васильева (в теории инвариантов
узлов конечного порядка используется также первое движение Васильева, состоящее
в обмене двух соседних концов дуг), мы будем называть его просто движением Васильева.

Далее в работе мы используем занумерованные дуговые диаграммы.
При сдвигах нумерация дуг сохраняется. При движении Васильева она ведет себя, как показано на Рис.~\ref{vm}.

\begin{figure}[H]

    \centering
    \includegraphics[scale=0.6]{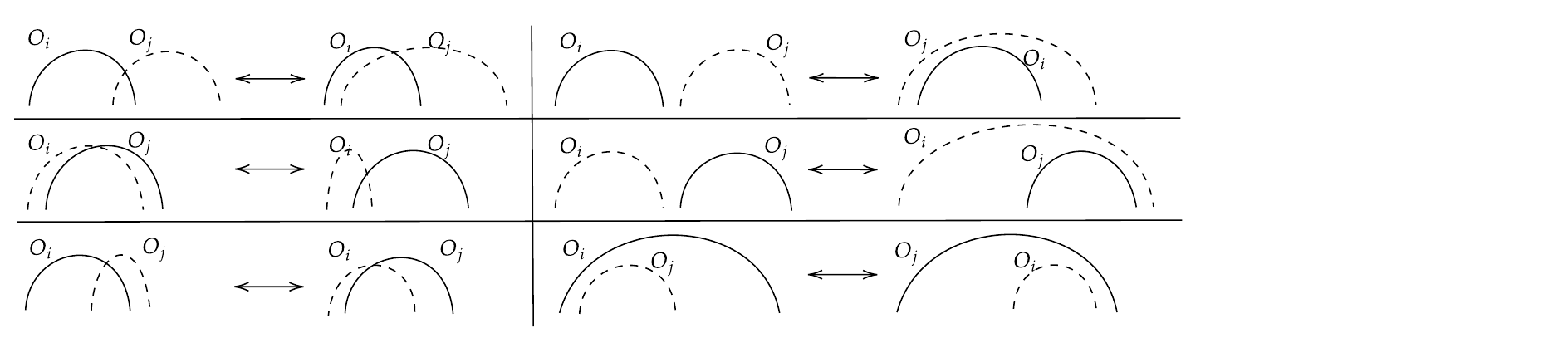}
    \caption{Различные варианты второго движения Васильева}
    \label{vm}
\end{figure}
Таким образом, по дуговой диаграмме с занумерованными дугами мы определяем нумерацию на диаграмме, полученной из данной
произвольной последовательностью движений Васильева.

 Пара точек пространства $\mathcal{R}_g(1^{2g})$, дуговые диаграммы которых  могут быть получены друг из друга сдвигами, соединены непрерывным путем в слое изопериодического слоения.

\subsubsection{Матрица периодов и матрица движений Васильева}

Разность значений интеграла $1$-формы $\psi$ от начальной точки~$Z$ до нуля $\mathcal{O}_i$ этой $1$-формы по двум негомотопным между собой путям в поверхности $X$, не пересекающим ребра графа~$G_\psi$,
обозначим через  $l_i$;  по определению, это длина дуги~$O_i$ дуговой диаграммы дифференциала~$\psi$.

Пусть $\mathcal{D}$ --- дуговая диаграмма с произвольной нумерацией дуг для вещественно-нормированного дифференциала~$\psi$ на кривой~$X$ рода~$g$, имеющего максимально несоизмеримые периоды.  Длины $l_1,\dots,l_{2g}$ ее произвольно занумерованных дуг порождают подгруппу $L\cong \mathbb{Z}^{2g}$ периодов $1$-формы $\psi$
 в~${\mathbb R}$. Любой другой занумерованной дуговой диаграмме $\tilde{\mathcal{D}}$, такой, что соответствующий дифференциал имеет максимально-несоизмеримые периоды  $\tilde{l}_1,\dots,\tilde{l}_{2g}$, порождающие тот же модуль периодов $\langle \tilde{l}_1,\dots,\tilde{l}_{2g}\rangle = L$,  можно однозначно сопоставить целочисленную матрицу размера $2g\times 2g$,  задающую замену базиса $\{l_i\}$ на $\{\tilde{l}_i\}$ в модуле $L$. Такую матрицу будем называть \emph{матрицей периодов занумерованной дуговой диаграммы} $\tilde{\mathcal{D}}$ \emph{относительно  занумерованной диаграммы} $\mathcal{D}$.

Задав нумерацию на  дуговой диаграмме $\mathcal{D}_{\psi}$, мы
по всякой последовательности движений Васильева  восстанавливаем нумерацию
дуг на новой диаграмме ${\mathcal{D}}_{\tilde{\psi}}$.
Матрицу, задающую преобразование на элементах базиса, соответствующую движению Bасильева,
будем называть \emph{матрицей этого движения}.
Сдвиги сохраняют матрицу периодов дуговой диаграммы. Поэтому для диаграммы, полученной из данной
последовательностью сдвигов и движений Васильева, матрицу периодов можно представить как композиицию
умножений исходной матрицы периодов на  матрицы отдельных движений.

Заметим, что матрица любого движения Васильева действует на матрице периодов элементарным преобразованием,
следовательно не меняет ее определитель.

\textbf{g-караваны.}
Дуговая диаграмма $\mathcal{D}_{\psi}$, состоящая из~$g$ пар
пересекающихся дуг,
таких, что дуги различных пар между собой не пересекаются и никакая пара не располагается на интервале, ограниченном концами третьей дуги, называется $g$-\emph{караваном}. Такие дуговые диаграммы являются допустимыми.

Если набор длин дуг $g$-каравана $l_1,\dots,l_{2g}$ максимально несоизмерим, то такой $g$-караван соответствует множеству дифференциалов с максимально несоизмеримым модулем периодов $\langle l_1,\dots,l_{2g}\rangle\cong L$.  Cимплектическая форма $\Omega$ на $L$ в выбранном упорядоченном базисе (длины произвольно занумерованных дуг $\mathcal{D}_{\psi}$)
  задается матрицей пересечений $I(\mathcal{D_{\psi}})$, примеры для случая~$2$-караванов показаны на рис.~\ref{ex2}. Длины дуг на рисунке обозначены буквами $x,y,z,t$. Заметим, что верхний и нижний~$g$-караваны соответствуют одной группе периодов, но различным ее поляризациям.

\begin{figure}[H]
    \centering
    \includegraphics[scale=1]{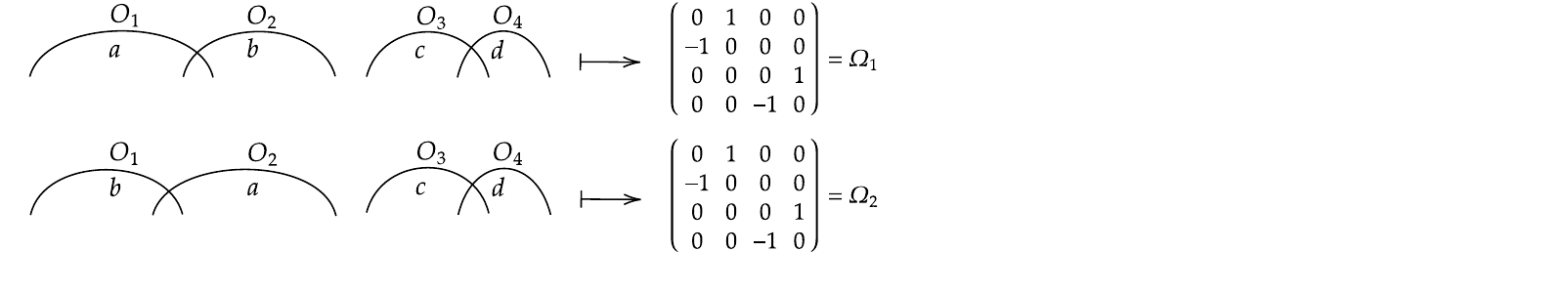}
    \caption{ $2$-караваны и соответствующие им матрицы форм пересечений }
    \label{ex2}
\end{figure}

Следующее утверждение доказано в работе~\cite{LK} (предложение 10).
\begin{Prop}
    \label{car}
Посредством сдвигов и движений Васильева любую дуговую диаграмму можно преобразовать в $g$-караван.
\end{Prop}

Каждому $g$-каравану $\mathcal{D}_{\psi}$ сопоставим вектор длины~$2g$, элементами которого являются длины дуг диаграммы $D_{\psi}$, перечисленные в порядке расположения левых концов дуг на диаграмме слева направо. Этот вектор назовем \emph{вектором периодов} дифференциала $\psi$, обозначим его $P_{\psi}$.

\begin{Prop}
\label{incl}
Для любых двух точек  в пространстве $\mathcal{R}_g(1^{2g})$, таких, что их дуговые диаграммы $\mathcal{D}_1,\mathcal{D}_2$ имеют вид $g$-караванов с векторами периодов $P_1,P_2$ и $\mathcal{D}_2$ получается  из $\mathcal{D}_1$ какой-то последовательностью сдвигов дуг и  движений Васильева, существует матрица $M\in Sp_{2g}(\mathbb{Z})$, такая, что $P_1=MP_2$.
\end{Prop}
\begin{proof}
 Занумеруем дуги в $\mathcal{D}_1$ и в $\mathcal{D}_2$ в порядке расположения левых концов дуг слева направо. В качестве
 искомой матрицы $M$ можно выбрать произведение двух матриц --- матрицы последовательности движений Васильева
 (записанной в соответствии с выбранной нумерацией дуг) и матрицы перестановки,
 осуществляющей переход от нумерации, индуцированной последовательностью движений Васильева,
 к стандартной.
 Матрица пересечений $I(\mathcal{D}_1)$ определяет симплектическую форму на группе периодов в базисе~$P_1$. Поскольку матрица пересечений $I(\mathcal{D}_2)$ имеет такой же вид в базисе~$P_2$,
 матрица, отвечающая замене базиса в группе периодов, является симплектической.
\end{proof}

Предложения \ref{car} и \ref{incl} показывают, в частности,  что если точки в $\mathcal{R}_g{(1^{2g})}$ имеют дуговые диаграммы, которые могут быть получены друг из друга при помощи сдвигов и движений Васильева, то они принадлежат одному подпространству $\mathcal{A}_L(1^{2g})$.

\section{Доказательство основной теоремы}\label{s3}

\subsection{Род 1}

Докажем теорему~\ref{teorema} для случая поверхности $X$ рода $1$. Все допустимые дуговые диаграммы в этом случае являются $1$-караванами. Зафиксируем нумерацию $1$-каравана, где левая дуга имеет номер $1$.

Введем следующее отношение эквивалентности~$\sim$ на упорядоченных парах вещественных чисел:~$(x,y)\sim (-y,x)$. Тогда в каждом ненулевом классе эквивалентности $4$ элемента:

\begin{equation}
\label{relation}
    (x,y)\sim (-y,x)\sim(y,-x)\sim(-x,-y).
\end{equation}
Заметим, что в каждом классе эквивалентности $[(x,y)]$ с ненулевыми $x$, $y$ есть ровно один \emph{положительный элемент} (пара, где оба числа положительные).

 Назовем~$1$-караван   $\mathcal{D}_{\psi}$ \emph{ассоциированным} классу~$[(x,y)]$, если его вектор периодов $P_{\psi}$ задан положительным элементом класса эквивалентности $[(x,y)]$.
Все~$1$-караваны, ассоциированные  одному и тому же классу, могут быть получены друг из друга сдвигами дуг, без применений движений Васильева

\begin{Prop}
\label{plus}
Если~$x$ и~$y$ несоизмеримы, то
из~$1$-каравана, ассоциированного классу $[(x,y)]$, можно сдвигами и движениями Васильева получить~$1$-караван, ассоциированный классу  $[(x+y,y)]$.
\end{Prop}
\begin{proof}
Рассмотрим два случая:\begin{itemize}
    \item $x$ и $y$ имеют одинаковый знак, значит $(x,y)\sim (|x|,|y|)$. Движение Васильева, примененное к правому концу левой дуги, дает пару  $(|x|+|y|,|y|)\sim(x+y,y)$.

    \item $x$ и $y$ разного знака, значит $(x,y)\sim (|y|,|x|)$. Сдвинем правую дугу влево и применим к ней движение Васильева. В случае $|x|>|y|$, мы получаем $(|y|,|x|-|y|)=(|y|,|x+y|)\sim (x+y, y)$.  Если $|y|>|x|$, получаем $(|y|-|x|,|y|)=(|y+x|,|y|)\sim (x+y,y)$.
\end{itemize}
\end{proof}
\begin{Prop}
\label{plus2}
Для любых несоизмеримых~$x$ и~$y$
из  дуговой диаграммы, ассоциированной классу $[(x,y)]$, можно движениями Васильева получить диаграммы, ассоциированные  классам  $[(x\pm y,y)]$  и  $[(x,y \pm x)]$.
\end{Prop}
\begin{proof}
Требуется получить диаграммы, соответствующие  четырем классам:
\begin{itemize}
     \item $[(x+y,y)]$. Уже рассмотрен в предложении \ref{plus}.
    \item $[(x-y,y)]$. По предложению \ref{plus} из такой пары можно получить $1$-караван, ассоциированный классу $[(x-y+y,y)]$, который равен классу $[(x,y)]$. Значит, обратными движениями из $1$-каравана, ассоциированного классу $[(x,y)]$, можно получить $1$-караван, ассоциированный классу $[(x-y,y)]$.

    \item   $[(x,y\pm x)]$. Класс $[(x,y)]$ содержит элемент $(y,-x)$, значит по предыдущему пункту можно получить $1$-караван, ассоциированный классу $[(y\pm x, -x)]$, который равен классу $[(x, y\pm x)]$.

\end{itemize}
\end{proof}

\begin{lemma}\label{2}
Подпространство  $\mathcal{A}_L(1,1)$ в $\mathcal{R}_g(1,1)$ линейно связно для каждого поляризованного модуля периодов~$L$ ранга~$2$.
\end{lemma}
\begin{proof}

Предложение~\ref{plus2} показывает, что из диаграммы, ассоциированной классу $[(x,y)]$, можно,
двигаясь по непрерывному пути в $\mathcal{A}_L(1,1)$, получить диаграммы, ассоциированные классам $[(x,y)\begin{psmallmatrix}
1 & \pm 1 \\
0 & 1
\end{psmallmatrix}]$ и $[(x,y)\begin{psmallmatrix}
1 & 0 \\
\pm 1 & 1
\end{psmallmatrix}]$.

Следовательно, мы можем получить из диаграммы, ассоциированной классу $[(x,y)]$,
диаграмму, ассоциированную $[(x,y)M]$, для любой $M\in  Sp_2(\mathbb{Z})$,  так как
четыре матрицы $\begin{psmallmatrix}
1 & \pm 1 \\
0 & 1
\end{psmallmatrix}
,
\begin{psmallmatrix}
1 & 0 \\
\pm 1 & 1
\end{psmallmatrix}$ порождают группу $ Sp_2(\mathbb{Z})=SL_2(\mathbb{Z})$.
Это означает, что любые два $1$-каравана $\mathcal{D}_1, \mathcal{D}_2$, построенные по дифференциалам $\psi_1,\psi_2$
с поляризованным модулем периодов $L$, могут быть преобразованы друг в друга сдвигами и движениями Васильева. Следовательно, любая пара точек в $\mathcal{A}_L(1,1)$ связана непрерывным путем. Лемма доказана.
\end{proof}

\subsection{Род 2}
\label{sec2}
Теперь докажем теорему~\ref{teorema} в случае, когда род $g$ римановой поверхности $X$ равен $2$.
В этом случае дуговая диаграмма состоит из~$2g=4$ дуг.

\begin{lemma}
\label{g2}
Подпространство  $\mathcal{A}_L(1^4)$ в $\mathcal{R}_2(1^4)$ линейно связно дла каждого поляризованного модуля периодов~$L$ ранга~$4$.
\end{lemma}
\begin{proof}
Любую допустимую дуговую диаграмму в $\mathcal{R}_2(1^4)$ можно преобразовать в $2$-караван с помощью сдвигов и движений Васильева по предложению \ref{car}. Проверим теперь, что  любые два $2$-каравана $\mathcal{D}_1,\mathcal{D}_2$, построенные по точкам в $\mathcal{A}_L(1^4)$, можно преобразовать друг в друга. 
Векторы периодов $P_1,P_2$ любых двух $2$-караванов в $\mathcal{A}_L(1^4)$ связаны симплектическим преобразованием.
Начнем с выбора набора элементов группы $Sp_4(\mathbb{Z})$,
порождающих ее.
\begin{Prop}
\label{ab}
Следующий набор элементов и обратных к ним порождает группу $Sp_4(\mathbb{Z})$: \\
\\ $A_1=\begin{bmatrix}
1& 1&0&0 \\
0 & 1&0 &0\\
0&0&1&0\\
0&0&0&1
\end{bmatrix}
\quad B_1=\begin{bmatrix}
1& 0&0&0 \\
 1 & 1&0 &0\\
0&0&1&0\\
0&0&0&1
\end{bmatrix} \quad A_2=\begin{bmatrix}
1& 0&0&0 \\
0 & 1&0 &0\\
0&0&1&1\\
0&0&0&1
\end{bmatrix}
\quad B_2=\begin{bmatrix}
1& 0&0&0 \\
0 & 1&0 &0\\
0&0&1&0\\
0&0& 1&1
\end{bmatrix}\quad C=
\begin{bmatrix}
1 & 0&0&0 \\
0 & 1 &0&1\\
-1&0&1&0\\
0&0&0&1
\end{bmatrix}\quad $
\end{Prop}

\begin{proof}
Мы доказываем это утверждение в предложении \ref{basis}  в более общем виде, для произвольного $g$.
\end{proof}
Сначала проверим, что  $2$-караван $\mathcal{D}$ с вектором периодов $P$ можно преобразовать сдвигами и движениями Васильева в $2$-караваны с векторами периодов  $CP,C^{-1}P$, если $CP>0,C^{-1}P>0$.\\
\begin{Prop}
Для  $2$-каравана $\mathcal{D}$ с вектором периодов  $P=(x,y,z,t)$ и матрицы $C=
\begin{bmatrix}
1 & 0&0&0 \\
0 & 1 &0&1\\
-1&0&1&0\\
0&0&0&1
\end{bmatrix}\quad $:
\begin{itemize}
    \item если $z>x$, то для $\mathcal{D}$ существует последовательность сдвигов и движений Васильева, переводящая его в $2$-караван с вектором периодов $P_1=CP$;
    \item  если $y>t$, то для $\mathcal{D}$ существует последовательность сдвигов и движений Васильева, переводящая его в $2$-караван с вектором периодов $P_2=C^{-1}P$.
\end{itemize}
\label{class}
\end{Prop}
\begin{proof}
Соответствующая последовательность преобразований для матрицы $C$ показана на рис.\ref{b1}. Дуга, которая заменяется при движении Васильева, на рисунке показана пунктиром. Горизонтальные стрелки указывают направление сдвига дуги также выделенной пунктиром. Последовательность для матрицы $C^{-1}$ получается обратными преобразованиями.
\end{proof}

\begin{figure}
    \centering
    \includegraphics[scale=0.25]{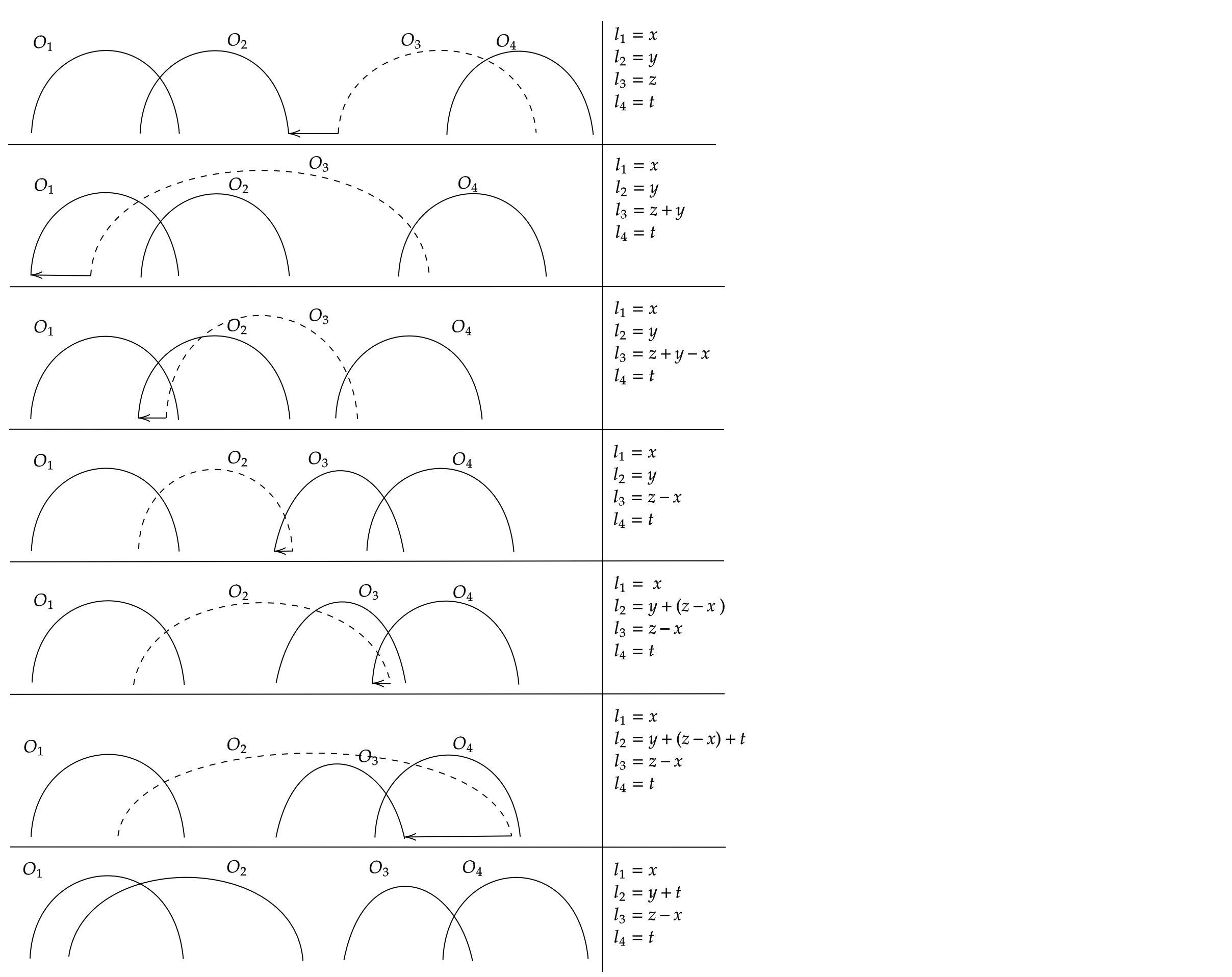}
    \caption{Последовательность движений Васильева $2$-караванов, реализующих  преобразование $C$}
    \label{b1}
\end{figure}

Отметим, что каждая из матриц $A_1$ и $B_1$ коммутирует с каждой из матриц $A_2$ и $B_2$;
в свою очередь, матрица $C$ коммутирует с матрицами $B_1$ и $A_2$.

Периоды каждого $2$-каравана с решеткой периодов~$L$ задают разложение решетки периодов~$L$ в прямую сумму двух $2$-мерных симплектических подрешеток $L=L_1\oplus L_2$, первая из которых натянута на периоды первого
$1$-каравана, вторая --- на периоды второго.

Введем отношение эквивалентности~$\sim$ на упорядоченных четверках ненулевых вещественных чисел
\begin{equation}(x,y,z,t)\sim (-y,x,z,t)\mbox{ и } (x,y,z,t)\sim (x,y,-t, z).
 \label{rel2}
\end{equation}Тогда в каждом классе эквивалентности $16$ элементов  и ровно один элемент из них \emph{положительный}, то есть все числа в нем больше нуля. $2-$караван с вектором периодов, заданным положительным элементом какого-то класса, будем называть \emph{ассоциированным этому классу эквивалентности}.

Теперь покажем, что $4$-караван c вектором периодов $P_1=(x,y,z,t)$  можно сдвигами и движениями Васильева преобразовать в $4$-караван, ассоциированный классу $[(x,y,z,t)M]$, для каждой матрицы $M\in Sp_4(\mathbb{Z})$.
Матрицы $A_i, B_i$ действуют прибавлением и вычитанием внутри упорядоченных пар $(x,y), (z,t)$. Поэтому, чтобы попасть в класс, которому принадлежит четверка чисел $[(x,y,z,t)A_i]$, достаточно реализовать последовательность движений на одной из пар зацепленных дуг, так же, как в случае рода $1$, описанном в предложении \ref{plus2}. Теперь докажем аналогичное утверждение для порождающих матриц $C, C^{-1}$.
\begin{Prop}
Из  дуговой диаграммы, ассоциированной классу $[(x,y,z,t)]$,
 можно движениями Васильева получить диаграммы, ассоциированные классам  $[(x,y,z,t)C], [(x,y,z,t)C^{-1}]$,  где $C=
\begin{bsmallmatrix}
1 & 0&0&0 \\
0 & 1 &0&1\\
-1&0&1&0\\
0&0&0&1
\end{bsmallmatrix}$.
\label{class}
\end{Prop}
\begin{proof}
Поскольку преобразование, обратное движению Васильева,  также  является
движением Васильева, достаточно доказать предложение только для класса $[(x,y,z,t)C]$.  Требуется показать, что для класса $[(x,y,z,t)]$ существуют движения Васильева, переводящие ассоциированный с ним $2$-караван в $2$-караван, ассоциированный  с классом $[(x,y+t,z-x,t)]$.
Рассмотрим четыре случая:
\begin{itemize}
    \item $x>0,t>0$:
      Следуя рассуждению выше, мы можем преобразовать такой $2$-караван, изменив  величины $y,z$  в соответствии с действием матриц $B_1, A_2$. Получим $2$-караван, ассоциированный  классу $(x,y',z',t)>0$, где $y'=y+n_1x>0$ и $z'=z+n_2t>x>0$, для подходящих $n_1,n_2$. Затем применим к получившемуся $2$-каравану набор движений, показанный на рисунке \ref{b1}.  Получим $2$-караван, ассоциированный классу $[(x,y'+t,z'-x,t)]$. Теперь применим последовательность преобразований дуговой диаграммы, соответствующую преобразованиям, заданным матрицами $B_1^{-1}, A_2^{-1}$, чтобы получить желаемую комбинацию $(x,y+t,z-x,t)$.
\item $x<0,t>0$: в этом же классе находится элемент
$ (-x,-y,z,t)$, удовлетворяющий условиям пункта 1. Значит, реализовав обратные преобразования к преобразованиям в  пункте 1, можно получить  класс $[(-x,-y-t,z-x,t)]$. Нужный нам элемент также принадлежит этому классу:
$(x,y+t,z-x,t)\sim (-x,-y-t,z-x,t)$
\item
$x>0,t<0$:
в этом же классе находится элемент
$ (x,y,-z,-t)$, удовлетворяющий условиям пункта  1. Значит, реализовав обратные преобразования к преобразованиям в  пункте 1, можно получить  класс $[(x,y+t,-z+x,-t)]$. Нужный нам элемент также принадлежит этому классу:
$(x,y+t,-z+x,-t)\sim (x,y+t,z-x,t)$
\item
$x<0,t<0$:
в этом же классе находится элемент
$ (-x,-y,-z,-t)$, удовлетворяющий условиям пункта 1. Значит, реализовав преобразования в  пункте 1, можно получить  класс $[(-x,-y-t,-z+x,-t)]$. Нужный нам элемент также принадлежит этому классу:
$(-x,-y-t,-z+x,-t)\sim (-x,y+t,z-x,t)$.\\

\end{itemize}
\end{proof}
\end{proof}
\subsection{Случай произвольного рода}
Теперь мы можем доказать Теорему~\ref{teorema} для случая произвольного рода $g$.
\begin{proof}[Доказательство Теоремы~\ref{teorema}]
Утверждение доказано для $g=1,2$ в леммах \ref{2} и \ref{g2}. В старших родах также достаточно показать, что любые два  $g$-каравана, векторы  периодов которых связаны симплектическим преобразованием, могут быть получены друг из друга с помощью движений Васильева.

Подходящий набор порождающих для группы $Sp_{2g}(\mathbb{Z})$
задается набором матриц, естественно обобщающим на произвольную размерность матрицы $A_1,A_2,B_1,B_2$, $C$ из утверждения \ref{ab}. При данном $g$ матрица $A_k$, $k=1,\dots,g$, размера $2g\times 2g$ задана как блочно-диагональная матрица с  $g$ блоками размера $2\times 2$, где все блоки, кроме $k$-го, это единичные матрицы, а блок с индексом $k$ равен $\begin{bsmallmatrix}1 & 1\\
0&1
\end{bsmallmatrix}$.

Матрицы из второго семейства, $B_k$, $k=1,\dots,g$,
определяются так же, с заменой блока с индексом~$k$
блоком $\begin{bsmallmatrix}1 & 0\\
1&1  \end{bsmallmatrix}$. Матрицы семейства~$C_k$,
$k=1,2,3,\dots,g-1,$
при фиксированном $g$ определим следующим образом: это блочно-диагональные матрицы размера $2g\times 2g$, у которых $g-2$ блоков это единичные матрицы размера $2\times 2$ и блок с индексом $k$ размера $4\times 4$ имеет вид $\begin{bsmallmatrix}
1 & 0&0&0 \\
0 & 1 &0&1\\
-1&0&1&0\\
0&0&0&1
\end{bsmallmatrix}$.

Проверим, что матрицы $A_i,B_i,C_i$ и обратные к ним порождают группу  $Sp_{2g}(\mathbb{Z})$.
\begin{Prop}
\label{basis}
Матрицы $A_1^{\pm 1},\dots,A_{g}^{\pm 1},B_1^{\pm 1},\dots,B_g^{\pm 1},C^{\pm1}_1,\dots,C^{\pm1}_{g-1}$ порождают группу $Sp_{2g}(\mathbb{Z})$.
\end{Prop}
\begin{proof}
Покажем, что из каждой cимплектической матрицы $U\in Sp_{2g}(\mathbb{Z})$, действуя слева и справа матрицами из семейств $A,B,C$, можно получить единичную матрицу.

Обозначим строки матрицы $U={(u_{i,j})}$ через $s_i$, а ее столбцы через $c_j$.
 Матрицы $A_k$ при умножении слева действуют прибавлением или вычитанием строк  внутри  упорядоченных пар строк $(s_{2k-1},s_{2k}), k\geq 1$. При умножении слева на матрицу $C_k$ строки  $(s_{2k-1},s_{2k}, s_{2k+1})$ матрицы $U$ заменяются соответственно на  $(s_{2k}+s_{2k+2}, s_{2k+1}-s_{2k-1})$.

 Используя перестановку пар, которая осуществляется композицией матриц типа $A, B$, можно добиться того, чтобы $u_{1,1}\neq 0$. 
 Поскольку определитель равен $1$, существует пара строк, такая, что сделав их первой и второй,
 мы получим $u_{1,1}\neq 0$ и если $u_{2,1}\neq 0$, то $gcd(u_{1,1},u_{2,1})\neq 0$.
 Используя матрицы типа $A$, приведем первый столбец к виду $(gcd_{1,1},0,gcd_{2,1},\dots, gcd_{g,1},0)$, где $gcd_{i,1}$ это наибольший общий делитель элементов $u_{2i-1,1}, u_{2i,1}$ (рис.\ref{1-19}). Далее, используя перестановку строк внутри пары со сменой знака (которая получается композицией матриц типа $A$)  и действие матриц типа $B$, получим для $(0,0)$ в первом столбце для нижней пары строк и $(0,\pm gcd_{(g-1,g),1})$ для пары $(s_{2(g-1)-1}, s_{2(g-1)})$. Здесь через $ gcd_{(g-1,g),1}$ обозначен наибольший общий делитель $gcd(gcd_{g-1,1},gcd_{g,1})$. Продолжив аналогичным образом, можно получить значение $(0,0)$ в первом столбце для всех пар строк, кроме первой, для которой значения будут соответственно $(\pm gcd_{(1,\dots,2g),1},0)$. Здесь $gcd_{(1,\dots,2g),1}$ обозначает $gcd(gcd_{1,1},\dots,gcd_{g,1})$ и поскольку  определитель матрицы $U$ равен $1$, имеем $(\pm gcd_{(1,\dots,2g),1},0)=(\pm 1,0)$. Мы добились того, что первый столбец приобрел вид $(\pm 1,0,\dots,0)^T$.

Матрицы $A_k,B_k$ действуют на столбцах  умножением справа. Поэтому рассмотрев  последовательность преобразований,
 аналогичную последовательности на строках, можно привести первую строку к виду $(\pm 1,0, gcd_{1,(2,\dots,g)},0\dots)$. Здесь $gcd_{1,(2,\dots,g)}=gcd(\tilde u_{1,2},\dots, \tilde u_{2g,2})$. Применив $B_1$ $gcd_{1,(2,\dots,g)}$ раз, получим матрицу, у которой первый столбец по-прежнему равен  $(\pm 1,0,\dots,0)^T$, а первая строка имеет вид  $(\pm 1,0,\dots,0)$.

Поскольку все матрицы $U, A_i,B_j$ являются элементами группы $Sp_{2g}({\mathbb Z})$, второй столбец и вторая строка матрицы, полученной в результате преобразований,  имеют вид  $(0,\pm 1,0\dots)^T,(0,\pm 1,0,\dots) $.  Используя преобразования типа $A$, можно добиться положительных значений ($1$, вместо $\pm 1$). Согласно предположению индукции, оставшиеся столбцы и строки можно также привести к диагональному виду. Получаем единичную матрицу, что и требовалось.

\begin{figure}[H]
    \centering
    \includegraphics[scale=0.45]{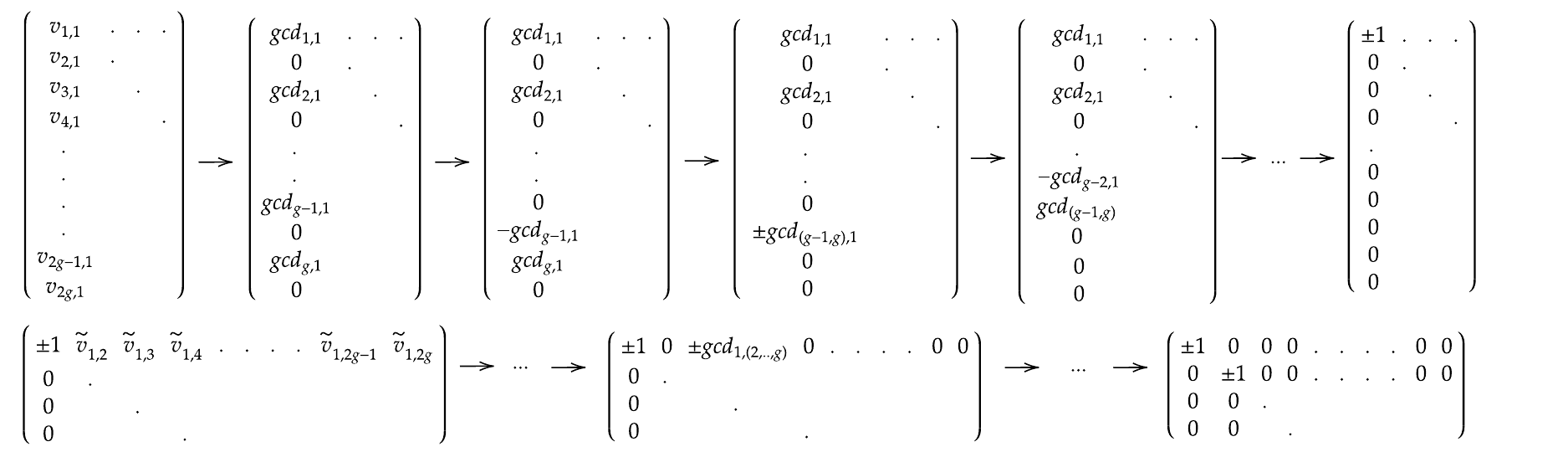}
    \caption{Порядок приведения матрицы к диагональной форме}
    \label{1-19}
\end{figure}
\end{proof}
\end{proof}

    \bibliographystyle{plain}
\bibliography{main}
\end{document}